\documentclass[a4paper]{amsart}
\usepackage{amsmath}
\usepackage{amssymb}
\usepackage{amsthm}

\usepackage[utf8]{inputenc}
\usepackage[T1]{fontenc}
\usepackage[charter]{mathdesign}
\usepackage{hyperref}
\hypersetup{
bookmarks=true,
pdfpagemode=UseNone,
pdfstartview=FitH,
pdfdisplaydoctitle=true,
pdftitle={A random pointwise ergodic theorem with Hardy field weights},
pdfauthor={Ben Krause and Pavel Zorin-Kranich},
pdflang=en-US
}
\usepackage{amsmath,amsthm,amssymb}
\theoremstyle{plain}

\newtheorem{theorem}[equation]{Theorem}

\newtheorem{proposition}[equation]{Proposition}
\newtheorem{lemma}[equation]{Lemma}

\numberwithin{equation}{section}

\usepackage[style=alphabetic,firstinits,maxnames=4,eprint=true,doi=false,isbn=false,backend=bibtex]{biblatex}
\def\woMR#1{\w@MR#1MR#1MR\relax}%
\def\w@MR#1MR#2MR#3\relax{#2}
\def\MR@URL#1 #2\relax{http://www.ams.org/mathscinet-getitem?mr=#1}
\def\MRnumber#1{\href{\MR@URL#1 \relax}{\nolinkurl{\woMR#1}}}% make a link to mathscinet
\DeclareFieldFormat{mrnumber}{\mkbibacro{MR}\addcolon\space\MRnumber{#1}}
\newbibmacro*{doi+eprint+url}{%
\iftoggle{bbx:doi}
{\printfield{doi}}
{}%
\newunit\newblock
\iftoggle{bbx:eprint}
{\usebibmacro{eprint}}
{}%
\newunit\newblock
\iftoggle{bbx:url}
{\usebibmacro{url+urldate}}
{}%
\newunit\newblock
\printfield{mrnumber}}

\newbibmacro{string+urldoi}[1]{%
  \iffieldundef{url}{%
    \iffieldundef{doi}{%
        #1%
      }{% if doi defined
        \href{http://dx.doi.org/\thefield{doi}}{#1}%
    }%
  }{% if url defined
    \href{\thefield{url}}{#1}%
  }%
}
\DeclareFieldFormat[article,misc]{title}{\usebibmacro{string+urldoi}{\mkbibquote{#1}}}

\newcommand*{\Z}{\mathbb{Z}}
\newcommand*{\Q}{\mathbb{Q}}

\newcommand*{\N}{\mathbb{N}}
\newcommand*{\R}{\mathbb{R}}
\newcommand*{\RR}{\mathbb{R}}

\newcommand*{\E}{\mathbb{E}}

\newcommand{\lac}[1][\N]{\lfloor \rho^{#1} \rfloor}

\def\<{\left\langle}
\def\>{\right\rangle}

\begin{document}
\allowdisplaybreaks
\title{A random pointwise ergodic theorem with Hardy field weights}
\author{Ben Krause}
\address{UCLA Math Sciences Building\\
         Los Angeles, CA 90095-1555, USA}
\email{benkrause23@math.ucla.edu}
\author{Pavel Zorin-Kranich}
\address{Universit\"at Bonn\\
  Mathematisches Institut\\
  Endenicher Allee 60\\
  53115 Bonn\\
  Germany
}
\email{pzorin@uni-bonn.de}
\maketitle

\begin{abstract}
Let $a_n$ be the random increasing sequence of natural numbers which takes each value independently with probability $n^{-a}$, $0 < a < 1/2$, and let $p(n) = n^{1+\epsilon}$, $0 < \epsilon < 1$.
We prove that, almost surely, for every measure-preserving system $(X,T)$ and every $f \in L^1(X)$ the modulated, random averages
\[
\frac{1}{N} \sum_{n = 1}^N e(p(n)) T^{a_n(\omega)} f
\]
converge to $0$ pointwise almost everywhere.
\end{abstract}

\section{Introduction}
A sequence of integers $\{n_k\} \subset \Z$ is said to be \emph{universally $L^p$-good} if for every measure-preserving system $(X,\mu,T)$ and every $f \in L^p(X)$ the subsequence averages
\[
A_N^{\{n_k\}} f := \frac{1}{N}\sum_{k=1}^N T^{n_k}f
\]
converge pointwise almost everywhere.
In this language, Birkhoff's classical pointwise ergodic theorem \cite{0003.25602} states that the full sequence of integers is universally $L^1$-good.

Obtaining pointwise convergence results for rougher, sparser sequences is much more challenging.
For instance, Bourgain's Polynomial Ergodic Theorem \cite{MR1019960} states that the sequence $\{P(n)\}$, $P$ integer polynomial, is universally $L^p$-good for each $p>1$.
Note that $\{P(n)\}$ are zero-Banach-density subsequences of the integers; in fact, Bourgain used a probabilistic method to find \emph{extremely} sparse universally good sequences.
From now on $\{X_n\}$ will denote a sequences of independent $\{0,1\}$ valued random variables (on a probability space $\Omega$) with expectations $\sigma_n$.
The \emph{counting function} $a_n(\omega)$ is the smallest integer subject to the constraint
\[ X_1(\omega) + \dots + X_{a_n(\omega)}(\omega) = n.\]

\begin{theorem}[{\cite[Proposition 8.2]{MR937581}}]
Suppose
\[
\sigma_n = \frac{ (\log \log n)^{B_p}}{n}, \ B_p > \frac{1}{p-1}, \ 1 < p \leq 2.
\]
Then, almost surely, $\{a_{n}\}$ is universally $L^p$-good.
\end{theorem}

In the years to follow random sequences became a widely used model for pointwise ergodic theorems.
One indication at their amenability to analysis is LaVictoire's $L^{1}$ random ergodic theorem.
\begin{theorem}[{\cite{MR2576702}}]
Suppose $\sigma_n = n^{-a}$ with $0< a < 1/2$.
Then, almost surely, $\{a_{n}\}$ is universally $L^1$-good.
\end{theorem}
Here, by the strong law of large numbers, almost surely
\[ a_n(\omega)/ n^{\frac{1}{1-a}} \]
converges to a non-zero number.
For comparison,  it is known that the sequences of $d$-th powers, $d>1$ integer, are universally $L^{1}$ bad \cite{MR2331324,MR2788358}.

Random sequences have also been used as a model for multiple ergodic averages.
Frantzikinakis, Lesigne, and Wierdl recently showed the following.
\begin{theorem}[{\cite[Theorem 1.1]{MR3043589}}]
\label{thm:flw}
Suppose $\sigma_{n} = n^{-a}$, $0<a<1/14$.
Then, almost surely, $(a_{n})_{n}$ has the following property:
for every pair of measure preserving transformations $T,S$ on a probability space $X$ and any functions $f,g\in L^{\infty}(X)$ the averages
\[
\sum_{n=1}^{N} g(S^{n}x) f(T^{a_{n}}x)
\]
converge pointwise almost everywhere.
\end{theorem}
It is noted in their paper that the linear sequence of powers $S^{n}$ can likely be replaced by other deterministic sequences, but their method of proof did not seem to allow this.
In this article we prove a related result in which we are able to replace the linear sequence of powers by a sequence drawn from a more general class at the cost of weakening the result in several other respects.
More precisely, with $0 < \epsilon< 1$ arbitrary but fixed, suppose $p : \RR \to \RR$ is a \emph{logarithmico-exponential} function which satisfies
\begin{enumerate}
\item the second-order difference relationship
\[ p(x+y+z) - p(x+y) - p(x+z)+ p(x) = O(x^{\epsilon -1} yz) \]
for $x,y, z >0$ (``big-O'' notation is recalled in the section on notation below); and
\item for all $a(x) \in C \cdot \Q [x]$, the set of real constant multiples of rational polynomials, $\frac{|a(x) - p(x)|}{\log x} \to \infty$.
\end{enumerate}
Good examples of such functions are $p(x) = x^{1+\epsilon}$. We refer the reader to \cite{MR2145587} for a more complete discussion of \emph{logarithmico-exponential} functions; informally, these are all the functions one can get by combining real constants, the variable $x$, and the symbols $\exp, \ \log, \cdot,$ and $+$. (e.g.\ $x^{1/2} = \exp(1/2 \cdot \log x)$ and $x^\pi/\log \log x$ are both logarithmico-exponential.)

Our main result is the following
\begin{theorem}\label{MAIN}
Suppose $\sigma_{n}=n^{-a}$, $0 < a < 1/2$, and $p$ is as above.
Then, almost surely, the following  holds:

For each measure-preserving system $(X,\mu,T)$ and each $f \in L^1(X)$ the averages
\[
\frac{1}{N} \sum_{n=1}^N e(p(n)) T^{a_n(\omega)}f
\]
converge to zero pointwise almost everywhere (here and later $e(t):= e^{2\pi i t}$).
\end{theorem}
Pointwise ergodic theorems with exponential polynomial weights are collectively known as Wiener--Wintner type theorems, see e.g.\ \cite{MR1995517} for linear polynomials and \cite{MR1257033} for general polynomials.
If the random sequence $\{a_{n}\}$ is replaced by the linear sequence $\{n\}$ in Theorem~\ref{MAIN}, the result follows from the Wiener--Wintner theorem for Hardy field functions due to Eisner and the first author \cite{2014arXiv1407.4736E}.
However, note that the full measure sets in our result depend on the choice of $p$.
It would be interesting to remove this dependence.
Also, the second order difference relation in the hypothesis of Theorem~\ref{MAIN} can likely be replaced by a polynomial growth assumption; this would require an inductive application of van der Corput's inequality.

The structure of this paper is as follows:\\
In \textsection\ref{sec:preliminaries} we introduce a few preliminary tools, discuss our proof strategy, and reduce our theorem to proving Proposition \ref{key}; and \\
In \textsection\ref{sec:proof} we prove Proposition \ref{key}, thereby completing the proof of Theorem \ref{MAIN}.

\subsection{Acknowledgments}
We are grateful to Nikos Frantzikinakis, Rowan Killip, and Alexander Volberg for helpful conversations, and to Tanja Eisner for suggesting the collaboration.
We thank the Hausdorff Research Institute for Mathematics for hospitality during the Trimester Program ``Harmonic Analysis and Partial Differential Equations''.

\section{Preliminaries}
\label{sec:preliminaries}
\subsection{Notation and tools}

With $X_n, \ \sigma_n$ as above, we let $Y_n := X_n - \sigma_n$.

We will be dealing with sums of random variables, so we introduce the following compact notation:
\[ S_{N} =\sum_{n=1}^{N} X_{N} \ \text{ and } \ S_{M,N}=\sum_{n=M}^{N} X_{n}.\]
We also let
\[ W_N := \sum_{n=1}^N \sigma_n,\]
so that $W_N$ grows as $N^{1-a}$.

We will make use of the modified Vinogradov notation. We use $X \lesssim Y$, or $Y \gtrsim X$ to denote the estimate $X \leq CY$ for an absolute constant $C$. If we need $C$ to depend on a parameter,
we shall indicate this by subscripts, thus for instance $X \lesssim_\omega Y$ denotes
the estimate $X \leq C_\omega Y$ for some $C_\omega$ depending on $\omega$.

We also make use of big-O notation: we let $O(Y)$ denote a quantity that is $\lesssim Y$, and similarly $O_\omega(Y)$ a quantity that is $\lesssim_\omega Y$.

The main probabilistic input in our argument is the following special case of Chernoff's inequality.
\begin{lemma}[see e.g.\ \cite{MR2573797}]
\label{lem:chernoff}
Let $\{X_n\}$, $\{ \sigma_n \}$ be as above. There exists an absolute constant $c >0$ so that for each $A >0$,
\[ \mathbb{P} \Big( | S_N - W_N | \geq A \Big) \lesssim \max\Big \{ \exp \Big(-c \frac{A^2}{W_N} \Big), \exp(-c A) \Big\}.\]
Consequently,
\[
\mathbb{P} \Big( |S_{N} - W_{N}| \geq \frac12 W_N \Big)
\lesssim
\exp(-c W_N)
\lesssim
\exp(-c N^{1-a}).
\]
\end{lemma}

This also implies the following version of the law of large numbers:
\begin{equation}
\label{eq:lln}
S_{N}/W_{N} \to 1
\quad\text{almost surely.}
\end{equation}

We will also need the Hilbert space van der Corput inequality.
\begin{lemma}[{see e.g.\ \cite{MR3043589}}]
Let $\{ v_{n}\} $ be a sequence in a Hilbert space $H$ and $1\leq M\leq N$.
Then
\begin{equation}
\label{vdC}
\Big\| \sum_{n=1}^{N} v_{n} \Big\|^{2}\
\leq
2 \frac{N}{M} \sum_{n=1}^{N} \| v_{n} \|^{2} +  4 \frac{N}{M} \sum_{m=1}^{M} \Big| \sum_{n=1}^{N-m} \< v_{n+m}, v_{n} \> \Big|
\end{equation}
\end{lemma}

\subsection{Strategy}
In proving his Random Ergodic Theorem, LaVictoire showed \cite{MR2576702} that on a set of full probability, $\Omega' \subset \Omega$, the maximal function
\[
f \mapsto \sup_N \frac{1}{N} \sum_{n=1}^N T^{a_n(\omega)}|f|,
\]
is weakly bounded on $L^1(X)$.
In particular, for $\omega \in \Omega'$ the set of $f \in L^1(X)$ for which the averages
\[
\frac{1}{N} \sum_{n=1}^N e(p(n)) T^{a_n(\omega)} f
\]
tend to zero pointwise a.e.\ is closed in $L^1$.
Hence it will be enough to prove pointwise convergence for $f \in L^\infty(X)$.
Now, as observed in \cite{MR1325697}, for bounded functions it is enough to prove convergence along every lacunary sequence $\lac = (\lfloor \rho^k \rfloor, k\in\N)$, where $\rho>1$ is taken from a countable sequence converging to $1$.

We will fix some $\rho > 1$ throughout, and the averaging parameters $N$ are assumed to belong to $\lac$ unless mentioned otherwise.

We follow a similar plan to \cite{MR3043589}.
We will prove Theorem \ref{MAIN} by showing that almost surely, for every measure-preserving system $(X,\mu,T)$ and every $f\in L^{\infty}(X)$, the following chain of asymptotic equivalences holds $\mu$-almost everywhere:
\begin{align}
\frac{1}{N}\sum_{n=1}^N e(p(n)) T^{a_n}f
&\approx\label{sim:def}
\frac{1}{S_N}\sum_{n=1}^N X_n(\omega) e(p(S_n)) T^nf  \\
&\approx\label{sim:lln}
\frac{1}{W_N}\sum_{n=1}^N X_n(\omega) e(p(S_n)) T^nf \\
&\approx\label{sim:summable}
\frac{1}{W_N}\sum_{n=1}^N \sigma_n e(p(S_n)) T^nf \\
&\approx\label{sim:ergodic-av}
\bar{f} \cdot \frac{1}{W_N}\sum_{n=1}^N \sigma_n e(p(S_n)) \\
&\approx\label{sim:reverse}
\bar{f} \cdot \frac{1}{N} \sum_{n=1}^N e(p(n)) \\
&\approx\label{sim:zero}
0.
\end{align}
Here, the symbol $\approx$ means that the difference converges to $0$ as $N\to\infty$ and $\bar{f} := \lim_N \frac{1}{N} \sum_{n=1}^{N} T^n f$ is the projection of $f$ onto the invariant factor of $T$.

Let us now list the ingredients used to establish the above asymptotic equivalences.
\begin{enumerate}
\item[\eqref{sim:def}] holds because the right-hand side equals the left-hand side with $N$ replaced by $S_{N}$.
\item[\eqref{sim:lln}] holds by \eqref{eq:lln}.
\item[\eqref{sim:summable}] is the key to our argument.
We isolate this crucial step in the following
\begin{proposition}\label{key}
In the setting of Theorem~\ref{MAIN}, almost surely the following holds: for each measure-preserving system $(X,\mu,T)$, and each $f \in L^2(X)$, the sequence
\[
\Big\| \frac{1}{W_N} \sum_{n=1}^N Y_n(\omega) e(p(S_n)) T^nf \Big \|_{L^2(X)}^2
\]
is summable over lacunary $N$, and in particular
\[
\frac{1}{W_N} \sum_{n=1}^N Y_n(\omega) e(p(S_n)) T^nf \to 0 \quad \mu\text{-a.e.}
\]
\end{proposition}
\item[\eqref{sim:ergodic-av}] for averages with weights $\sigma_{n}$ follows by the partial summation formula
\[
\frac{1}{W_{N}} \sum_{n=1}^{N} \sigma_{n} a_{n}
=
\frac{N\sigma_{N}}{W_{N}} A_{N} + \sum_{M=1}^{N-1} \frac{M(\sigma_{M}-\sigma_{M+1})}{W_{N}} A_{M},
\quad
A_{N} = \frac{1}{N} \sum_{n=1}^{N} a_{n},
\]
from the following result on unweighted averages with $G=e\circ p$:
\begin{lemma}
\label{lem:ergodic-av}
Suppose $0<a<1$.
Then, almost surely, for every measure-preserving system $(X,\mu,T)$ and every $f\in L^{1}(X,\mu)$ pointwise $\mu$-a.e.\ we have
\[
\frac{1}{N}\sum_{n=1}^N G(S_n) T^nf
\approx
\bar{f} \cdot \frac{1}{N}\sum_{n=1}^N G(S_n)
\]
for every bounded function $G:\N\to\R$ as $N\to\infty$.
\end{lemma}
This is a slight abstraction from \cite[Lemma 2.2]{MR3043589}, where a different function $G$ was specified (but its special form not used in the proof).
For completeness, the proof is reproduced below.

\item[\eqref{sim:reverse}] follows by applying the above steps in reverse order, with $f = 1_X$; and
\item[\eqref{sim:zero}] reduces to a statement about trigonometric sums, namely $\frac{1}{N} \sum_{n=1}^N e(p(n)) \to 0$, which was proved in \cite[Theorem 1.3]{MR1269206}.
\end{enumerate}

\begin{proof}[Proof of Lemma~\ref{lem:ergodic-av}]
By the usual maximal ergodic theorem, for each fixed $\omega$ the set of $f$ for which asymptotic equivalence holds a.e.\ is closed in $L^{1}(X)$.
Since the equivalence is clear in the case $f=\bar f$ and in view of the splitting $L^{2}(X) = \{f=\bar f\} \oplus \overline{\{ Th-h, h\in L^{\infty}(X)\}}$, it suffices to consider the case when $f=Th-h$, $h\in L^{\infty}$, is a coboundary, so that in particular $\bar f = 0$.
Since $f\in L^{\infty}$ in this case, it suffices to obtain equivalence for $N\in\lac$ with $\rho>1$ fixed but arbitrary.

Summation by parts gives
\[
\frac{1}{N}\sum_{n=1}^N G(S_n) T^n(h-Th)
=
O(\|G\|_{\infty}/N) + \frac{1}{N} \sum_{n=1}^{N-1} (G(S_n) - G(S_{n+1})) T^n h.
\]
The first summand is deterministic and converges to $0$.
The second summand is $\mu$-a.e.\ bounded by
\[
2\|G\|_{\infty}\|h\|_{\infty} \frac{1}{N} \sum_{n=1}^{N-1} X_{n+1}
\leq
2\|G\|_{\infty}\|h\|_{\infty} \frac{S_{N}}{N},
\]
and this converges to $0$ almost surely in view of \eqref{eq:lln}.
\end{proof}
With this reduction complete, we now turn to the proof of Proposition \ref{key}.

\section{Proof of Proposition \ref{key}}
\label{sec:proof}
Throughout this section, we will view $0< \delta \ll 1$ as a (small) floating parameter, whose precise value will be fixed at the end of the proof; $0 < \nu = \nu(\delta) = O(\delta)$ will be used to denote (possibly different) parameters (all of which grow linearly in $\delta$); $0< \kappa = O(\delta)$ will be used similarly.

We begin with a criterion that guarantees that a bounded sequence $\{c_n\}$ is a good sequence of weights for a pointwise ergodic theorem along a lacunary sequence.

\begin{lemma}
\label{lem:crit}
Let $0<a<b<1$ and fix $\rho>1$.
Let $\{c_n\}$ be a bounded sequence such that the following holds:
\begin{equation}
\label{eq:c-sum}
\sum_{n=1}^{N} |c_{n}| \lesssim N^{1-a},
\quad N\in\lac,
\quad\text{and}
\end{equation}
\begin{equation}
\label{eq:c-corr}
\sum_{N\in\lac} N^{2a-1-b}
\sum_{m=1}^{N^{b}} \Big| \sum_{n=N^{1-\delta}}^{N-m} c_{n+m} \bar c_{n} \Big|
< \infty.
\end{equation}
Then for every measure-preserving system $(X,\mu,T)$ and $f\in L^{2}(X)$ we have
\[
\sum_{N\in\lac} \Big\| \frac1{N^{1-a}} \sum_{n=1}^{N} c_{n} T^{n}f  \Big\|_{L^2(X)}^{2} < \infty.
\]
%Consequently, $\frac{1}{N^{1-a}} \sum_{n=1}^N c_n T^nf \to 0$ as $N\to\infty$, $N\in\lac$.
\end{lemma}
\begin{proof}
Note that \eqref{eq:c-sum} with $N\in\lac$ implies \eqref{eq:c-sum} with $N\in\N$, and we obtain
\begin{align*}
\Big\| \frac1{N^{1-a}} \sum_{n=1}^{N^{1-\delta}} c_{n} T^{n}f  \Big\|_{L^2(X)} &\leq
\frac1{N^{1-a}} \sum_{n=1}^{N^{1-\delta}} |c_{n}| \| f \|_{L^2(X)}\\
&\lesssim
\frac1{N^{1-a}} N^{(1-\delta)(1-a)} \| f \|_{L^2(X)} \\
&=
N^{-\delta (1-a)} \| f \|_{L^2(X)},
\end{align*}
so we may replace the sum in the conclusion of the lemma by $\sum_{n=N^{1-\delta}}^{N}$.

Using van der Corput inequality \eqref{vdC} on the Hilbert space $H=L^{2}(X)$ with $M=N^{b}$, estimate
\begin{multline}
\label{eq:vdC1}
\Big\| \frac1{N^{1-a}} \sum_{n=N^{1-\delta}}^{N} c_{n} T^{n}f  \Big\|_{L^2(X)}^{2}\\
\lesssim
N^{2a-2} \frac{N}{N^{b}} \sum_{n=N^{1-\delta}}^{N} \| c_{n} T^{n}f \|_{L^{2}(X)}^{2}
+
N^{2a-2} \frac{N}{N^{b}} \sum_{m=1}^{N^{b}} \Big| \sum_{n=N^{1-\delta}}^{N-m} \int_{X} c_{n+m} T^{n+m} f \bar c_{n} T^{n} \bar{f} \Big|.
\end{multline}
The first term in \eqref{eq:vdC1} is bounded by
\[
N^{2a-2} \frac{N}{N^{b}} \sum_{n=N^{1-\delta}}^{N} | c_{n} |^{2} \|f\|_{L^2(X)}^{2},
\]
and by the assumption \eqref{eq:c-sum} and boundedness of $(c_{n})$ this is $O(N^{a-b})$.
By precomposing with $T^{-n}$, the second term in \eqref{eq:vdC1} is bounded by
\[
N^{2a-1-b} \sum_{m=1}^{N^{b}} \Big| \sum_{n=N^{1-\delta}}^{N-m} c_{n+m} \bar c_{n} \Big| |\<T^{m}f,f\>_{L^2(X)}|,
\]
and this is summable by the assumption \eqref{eq:c-corr}.
\end{proof}

\begin{proposition}
\label{prop:as-corr}
Let $p:\R\to\R$ be a function such that
\begin{equation}
\label{eq:p}
p(x+y+z) - p(x+y) = p(x+z) - p(x) + O(x^{\epsilon-1}yz)
\end{equation}
for $x,y,z>0$.
Let also $0<a<1/2$ and fix $\rho>1$.
Then there exists $b\in (a,1/2)$ such that, almost surely, the sequence $c_{n}=Y_{n} e(p(S_{n}))$ satisfies \eqref{eq:c-corr}.
\end{proposition}
\begin{proof}
By Fubini's theorem it suffices to show that the expectation
\[
N^{2a-1-b} \sum_{m=1}^{N^{b}} \underbrace{\E \Big| \sum_{n=N^{1-\delta}}^{N-m} Y_{n+m} e(p(S_{n+m})) Y_{n} e(-p(S_{n})) \Big|}_{=: I(m)}
\]
is summable along the lacunary sequence $N\in\lac$.
By Cauchy--Schwarz we have
\begin{equation}
\label{eq:pre-vdC2}
I(m)^{2}
\leq
\E \Big| \sum_{n=N^{1-\delta}}^{N-m} Y_{n} Y_{n+m} e(p(S_{n+m})-p(S_{n})) \Big|^{2}.
\end{equation}
Using the van der Corput inequality \eqref{vdC} with values in the Hilbert space $H=L^{2}(\Omega)$ and $R=N^{c}$, $0<c<1$ to be chosen later, we obtain the estimate
\begin{align*}
&I(m)^2 \leq I_1(m)^2 + I_2(m)^2 + I_3(m)^2 :=\\
&\quad \frac{N-m}{R} \sum_{n=N^{1-\delta}}^{N-m} \| Y_{n} Y_{n+m} e(p(S_{n+m})-p(S_{n})) \|_{L^{2}(\Omega)}^{2}\\
& +
\frac{N-m}{R} \Big| \E \sum_{n=N^{1-\delta}}^{N-2m} Y_{n+2m} Y_{n+m} Y_{n+m} Y_{n} e(p(S_{n+2m})-p(S_{n+m})-p(S_{n+m})+p(S_{n})) \Big| \\
& +
\frac{N-m}{R} \sum_{r=1, r \neq m}^{R} \Big| \E \sum_{n=N^{1-\delta}}^{N-m-r} Y_{n+r} Y_{n+r+m} Y_{n} Y_{n+m} e(p(S_{n+r+m})-p(S_{n+r})-p(S_{n+m})+p(S_{n})) \Big|.
\end{align*}
The task is now to show that, uniformly in $m \leq N^b$, we have
\[ I_j(m)^2 \lesssim N^{2 - 4a - \kappa} \text{ for each } j = 1,2,3, \]
for some $\kappa = \kappa(\delta,a,b,c) > 0$.

To this end we estimate the first term, $I_1(m)^2$, by
\[
\frac{N}{R} \sum_{n=N^{1-\delta}}^{N-m} \| Y_{n} \|_{L^{2}(\Omega)}^{2} \| Y_{n+m} \|_{L^{2}(\Omega)}^{2}
\]
by independence; this is bounded by
\[
\frac{N}{R} \sum_{n=N^{1-\delta}}^{N-m} \sigma_{n} \sigma_{n+m}
\lesssim
N^{1-c} N^{1-2a} < N^{2 - 4a - \kappa}
\]
provided we take $ 2 a < c < 1$.

We next turn to $I_2(m)^2$, which contributes at most
\[
\frac{N}{R} \E \sum_{n=N^{1-\delta}}^{N-2m} | Y_{n+m} Y_{n+2m} Y_{n} Y_{n+m} |,
\]
and by independence this is bounded by
\begin{align*}
\frac{N}{R} \sum_{n=N^{1-\delta}}^{N-2m} \E |Y_{n+m}|^{2} \cdot \E |Y_{n+2m}| \cdot \E |Y_{n}|
&\lesssim
N^{1-c} \sum_{n=N^{1-\delta}}^{N-2m} \sigma_{n} \sigma_{n+m} \sigma_{n+2m} \\
&\lesssim
N^{1-c} N^{1-3a+\nu} \\
&\lesssim N^{2 - 4a - \kappa},
\end{align*}
provided $c > 2a$ (from above) and $\nu = \nu(\delta) > 0$ is taken sufficiently small. ($\nu$ arises from the possibility that $3a > 1$, in which case we may take e.g.\ $\nu = (3a -1) \delta$.)

The contribution of this term is also acceptable.

It remains to estimate $I_3(m)^2$, which we write in the form
\[
I_3(m)^2
=
\frac{N-m}{R} \sum_{r=1, r \neq m}^{R} \Big| \E \sum_{n=N^{1-\delta}}^{N-m-r} Y_{n+r} Y_{n+r+m} Y_{n} Y_{n+m} e(p(S_{n+s+t})-p(S_{n+t})-p(S_{n+s})+p(S_{n})) \Big|,
\]
with $s=\min(r,m)$ and $t=\max(r,m)$.
To recover independence we apply \eqref{eq:p} with
\[
x=S_{n+t-1}, \ y=X_{n+t}, \ z=S_{n+t+1,n+t+s},
\]
to the first two summands in the argument of $e$.
This gives the estimate
\begin{align*}
&\frac{N-m}{R} \sum_{r=1, r \neq m}^{R} \Big| \E \sum_{n=N^{1-\delta}}^{N-m-r} Y_{n+r} Y_{n+r+m} Y_{n} Y_{n+m} e(p(S_{n+t-1}+S_{n+t+1,n+t+s})-p(S_{n+t-1})-p(S_{n+s})+p(S_{n})) \Big|\\
&+
\frac{N-m}{R} \sum_{r=1, r \neq m}^{R} \E \sum_{n=N^{1-\delta}}^{N-m-r} \Big| Y_{n+r} Y_{n+r+m} Y_{n} Y_{n+m} \min(O(S_{n+t-1}^{\epsilon-1} X_{n+t} S_{n+t+1,n+t+s}),1) \Big|.
\end{align*}
The main feature of this splitting is that the exponential in the first term does not depend on $X_{n+t}$, so $Y_{n+t}$ is independent from all other terms.
Therefore the first term vanishes identically.
The second term is estimated by
\begin{align*}
&\frac{N}{R} \sum_{r=1, r \neq m}^{R} \sum_{n=N^{1-\delta}}^{N-m-r} \E \Big( | Y_{n+r} Y_{n+r+m} Y_{n} Y_{n+m} | \cdot \min(S_{n+t-1}^{\epsilon-1} S_{n+t+1,n+t+s},1) \Big)\\
&\leq
\frac{N}{R} \sum_{r=1, r \neq m}^{R} \sum_{n=N^{1-\delta}}^{N-m-r} \E \Big( | Y_{n+r} Y_{n+r+m} Y_{n} Y_{n+m} | \cdot \min(S_{n+t-1}^{\epsilon-1} (S_{n+t+1,n+t+s-1}+1),1) \Big)
\end{align*}
By Lemma~\ref{lem:chernoff} this is bounded by
\[
\frac{N}{R} \sum_{r=1, r \neq m}^{R} \sum_{n=N^{1-\delta}}^{N-m-r} \E \Big( | Y_{n+r} Y_{n+r+m} Y_{n} Y_{n+m} | \cdot \min((W_{n+t-1})^{\epsilon-1} (S_{n+t+1,n+t+s-1}+1), 1_{E_{n+t-1}}) \Big),
\]
where $E_{n}$ is an exceptional set of measure $\lesssim \exp(-c n^{1-a})$.
At this point we estimate the minimum by a sum.
We consider first the non-exceptional part.
All remaining random variables are independent, so we get the estimate
\begin{align*}
&\frac{N}{R} \sum_{r=1, r \neq m}^{R} \sum_{n=N^{1-\delta}}^{N-m-r} \sigma_{n+r} \sigma_{n+r+m} \sigma_{n} \sigma_{n+m} (n+t-1)^{(1-a)(\epsilon-1)} \Big( \sum_{j=n+t+1}^{n+t+s-1}\sigma_{j}+1 \Big)\\
& \qquad \leq \frac{N}{R} \sum_{r=1, r \neq m}^{R} \sum_{n=N^{1-\delta}}^{N-m-r} n^{-4a} n^{(1-a)(\epsilon-1)} (n^{-a} N^{b}+1)\\
&\lesssim \frac{N}{R} \sum_{r=1, r \neq m}^{R} N^{b} \sum_{n=N^{1-\delta}}^{N-m-r} n^{-4a} n^{(1-a)(\epsilon-1)} n^{-a}\\
&\lesssim N^{2 - 4a + ( b- a  + \nu) + (1-a)(\epsilon -1)} \\
&\lesssim N^{2 - 4a - \kappa},
\end{align*}
provided that $b$ is taken sufficiently close to $a$ and $\delta$ is sufficiently small, since $(1-a)(\epsilon-1)<0$.

Finally, for the exceptional part we have superpolynomial decay in $N$.
\end{proof}

Thus we have verified that the assumption \eqref{eq:c-corr} of Lemma~\ref{lem:crit} holds almost surely in the setting of Proposition~\ref{key}.
The missing assumption \eqref{eq:c-sum} also holds almost surely because
\[
\sum_{n=1}^{N} |Y_{n}|
\leq
S_{N} + W_{N}
\]
and in view of \eqref{eq:lln}.
This completes the proof of Proposition~\ref{key} and hence of Theorem~\ref{MAIN}.

\printbibliography
\end{document}